\documentclass[11pt]{article}
\usepackage[T1]{fontenc}
\usepackage{amsmath}
\usepackage{graphicx}
\usepackage[hidelinks]{hyperref}
\usepackage{url}
\usepackage{amssymb}
\usepackage{amsthm}
\usepackage{a4wide}
\usepackage{xcolor}
\def\tp#1#2{\mathrm {Tp}_{#1}(#2)}
\def\meet{\land}
\def\Fraisse{Fra\"{\i}ss\' e}
\def\height#1#2#3{\mathrm {height}_{#1}(#2,\allowbreak #3)}
\def\str#1{\mathbf {#1}}
\def\Emb{\mathop{\mathrm{Emb}}\nolimits}

\newtheorem{theorem}{Theorem}[section]

\theoremstyle{definition}
\newtheorem{definition}[theorem]{Definition}

\newtheorem{problem}[theorem]{Problem}

\theoremstyle{remark}

\numberwithin{equation}{section}

\title{On Big Ramsey degrees of universal $\omega$-edge-labeled hypergraphs}

\author{
Jan Hubička\thanks{Department of Applied Mathematics (KAM), Charles University, Ma\-lo\-stranské~nám\v estí 25, Praha 1, Czech Republic E-mail: {\tt hubicka@kam.mff.cuni.cz}. Supported by project 25-15571S of  the  Czech  Science Foundation (GA\v CR)}
\and
Matěj Konečný\thanks{Institute of Algebra, TU Dresden, Dresden, Germany. E-mail: {\tt matej.konecny@tu-dresden.de}. Supported by a project that has received funding from the European Union (Project POCOCOP, ERC Synergy Grant 101071674).  Views and opinions expressed are however those of the authors only and do not necessarily reflect those of the European Union or the European Research Council Executive Agency. Neither the European Union nor the granting authority can be held responsible for them.}
\and
Stevo Todorcevic\thanks{Department of Mathematics, University of Toronto, 40 St George St., Toronto, ON,  Canada. E-mail: {\tt stevo@math.toronto.edu}. Supported by NSERC grants RGPIN-2019-455916 and RGPIN-2025-04386.}
\and
Andy Zucker\thanks{Department of Pure Mathematics, University of Waterloo, 200 University Ave W, Waterloo, ON, Canada. E-mail: {\tt a3zucker@uwaterloo.ca}. Supported by NSERC grants RGPIN-2023-03269 and DGECR-2023-00412.}
}

\makeatletter

\newcommand{\shorttitle}{\@title}

\def\@maketitle{%
  \newpage
  \begin{center}%
  \let \footnote \thanks
    {\small Proceedings of the 13th European Conference on Combinatorics, Graph Theory and Applications\\ EUROCOMB'25\\
    Budapest, August 25 - 29, 2025
    }
    \vskip 0.5em
    \rule{\linewidth}{0.04cm}
    \vskip 3.5em
    {\LARGE \textbf{\textsc{\@title}} \par}%
    \vskip 1.5em
    {\textbf{\textsc{(Extended abstract)}} \par}
    \vskip 2.5em%
    {\large
      \lineskip .5em%
      \begin{tabular}[t]{c}%
        \@author
      \end{tabular}\par}%
  \end{center}%
  \par
  }
\makeatother

\usepackage{amsthm}

\begin{document}

\thispagestyle{empty}
\maketitle

\begin{abstract}
	We show that the big Ramsey degrees of every countable universal $u$-uniform $\omega$-edge-labeled hypergraph are infinite for every $u\geq 2$.
	Together with a recent result of Braunfeld, Chodounsk\'y, de Rancourt,
	Hubi\v cka, Kawach, and Kone\v cn\'y  this finishes full
	characterisation of unrestricted relational structures with finite big
	Ramsey degrees.
\end{abstract}


\section{Introduction}
Let $A$ be a set and let $u$ be a positive integer. We denote by $\binom{A}{u}$ the set of all $u$-element subsets of $A$.
Given a countable set $L$ of \emph{labels}, an \emph{$L$-edge-labeled $u$-uniform hypergraph} (or simply an \emph{edge-labeled hypergraph}) is a pair $\str{A} = (A, e_\str{A})$, where $e_\str{A}$ is a function $e_\str{A} \colon \binom{A}{u} \to L$.
We call $A$ the \emph{vertex set} of $\str{A}$ and consider only finite and countably infinite vertex sets.  We say that $\str{A}$ is \emph{finite} if $A$ is finite.
We will view $u$-uniform hypergraphs as $\{0,1\}$-edge-labeled $u$-uniform hypergraphs (where the label 0 represents non-edges) and \emph{graphs} as $\{0,1\}$-edge-labeled 2-uniform hypergraphs.

Given $L$-edge-labeled $u$-uniform hypergraphs $\str{A} = (A, e_\str{A})$ and $\str{B} = (B, e_\str{B})$, an \emph{embedding $f\colon \str{A}\to\str{B}$} is an injective function $f\colon A\to B$ such that for every $E\in \binom{A}{u}$ we have $e_\str{A}(E) = e_\str{B}(f[E])$, where $f[E]=\{f(v):v\in E\}$. If $A\subseteq B$ and the inclusion map is an embedding, we call $\str{A}$ a \emph{substructure} of $\str{B}$.   We say that $\str{A}$ is \emph{homogeneous} if every isomorphism between finite substructures of $\str{A}$ extends to an automorphism of $\str{A}$, and $\str{A}$ is \emph{universal} if every countable $L$-edge-labeled $u$-uniform hypergraph embeds into $\str{A}$.  It is a well-known consequence of the \Fraisse{} theorem~\cite{Fraisse1953} that for every finite integer $u$ and finite or countable set $L$ there exists an up-to-isomorphism unique universal and homogeneous $L$-edge-labeled hypergraph $\str{R}^u_L$. Equivalently, $\str{R}^u_L$ can be characterised by the \emph{extension property}: For every $u$-uniform $L$-edge-labeled hypergraph $\str{B}$ and its finite substructure $\str{A}$, every embedding $\str{A}\to\str{R}^u_\omega$ extends to an embedding $\str{B}\to\str{R}^u_\omega$, see e.g.\ \cite{Hodges1993}. If $\mu$ is a probability measure on $L$ with full support, then letting $e_\mu\colon \binom{\omega}{u}\to L$ be randomly generated according to $\mu$, the structure $(\omega, e_\mu)$ is with probability $1$ isomorphic to $\str{R}^u_L$, and thus hypergraphs $\str{R}^u_L$ can be called \emph{random} countable edge-labeled hypergraphs. $\str{R}^2_{\{0,1\}}$ is known as the \emph{random graph} or \emph{Rado graph}~\cite{cameron1997random}. 
Given edge-labeled hypergraphs $\str{A}$ and $\str{B}$, we denote by $\Emb(\str{A}, \str{B})$ the set of all embeddings from $\str{A}$ to $\str{B}$. If $\str{C}$ is another edge-labeled hypergraph and $\ell\leq k< \omega$, we write $\str{C} \longrightarrow (\str{B})^\str{A}_{k,\ell}$ to denote the following statement:
\begin{quote}
    For every colouring $\chi \colon \Emb(\str{A}, \str{C}) \to \{1, \dots, k\}$ with $k$ colours, there exists an embedding $f \colon \str{B} \to \str{C}$ such that the restriction of $\chi$ to $\Emb(\str{A}, f(\str{B}))$ takes at most $\ell$ distinct values.
\end{quote}
For a countably infinite edge-labeled hypergraph $\str{B}$ and a finite substructure $\str{A}$ of $\str{B}$, the \emph{big Ramsey degree of $\str{A}$ in $\str{B}$} is the least number $D\in \omega$ (if it exists) such that $\str{B} \longrightarrow (\str{B})^\str{A}_{k,D}$ for every $k\in \omega$. We say that $\str{B}$ \emph{has finite big Ramsey degrees} if the big Ramsey degree of every finite substructure $\str{A}$ of $\str{B}$ exists.

In 1969 Laver introduced a proof technique which shows that $\str{R}^2_L$ has finite big Ramsey degrees for every finite set $L$~\cite{devlin1979,erdos1974unsolved,todorcevic2010introduction}. This was refined by Laflamme, Sauer, and Vuksanovic~\cite{Laflamme2006} to precisely characterise the big Ramsey degrees of these structures. Finiteness of big Ramsey degrees of $\str{R}^3_{\{0,1\}}$ was announced at Eurocomb 2019 by Balko, Chodounsk\'y, Hubi\v cka, Kone\v cn\'y, and Vena~\cite{Hubickabigramsey} with a proof published in 2020~\cite{Hubicka2020uniform}. In 2024, Braunfeld, Chodounsk\'y, de Rancourt, Hubi\v cka, Kawach, and Kone\v cn\'y~\cite{braunfeld2023big} extended the proof to arbitrary finite $u>0$ and finite $L$ and generalised the setup to model-theoretic $L'$-structures where $L'$ is a (possibly infinite) relational language containing only finitely many relations of every given arity $a>1$.
Answering Question 7.5 of \cite{braunfeld2023big} we show that the assumption about finiteness of $L$ as well as the above assumption about language $L'$ is necessary:
\begin{theorem}
	\label{thm:main}
	Let $u>1$ be finite and let $\str{A}$ be any $\omega$-edge-labeled $u$-uniform hypergraph with 2 vertices.  Then $\str{A}$ does not have finite big Ramsey degree in $\str{R}^u_\omega$.
\end{theorem}
It is known that the big Ramsey degrees of $\str{R}^1_\omega$ are finite~\cite{braunfeld2023big}.  It is also easy to show:
\begin{theorem}
	Let $u>1$ be finite and $\str{A}$ be the $\omega$-edge-labeled $u$-uniform hypergraph with 1 vertex.  Then the big Ramsey degree of $\str{A}$ in $\str{R}^u_\omega$ is 1.
\end{theorem}
Consequently, our result concludes the characterisation of unrestricted structures with finite big Ramsey degrees (see~\cite{braunfeld2023big} for precise definitions).
Our proof introduces a new technique that complements the existing arguments for infinite lower bounds which can be divided into three types:
Counting number of oscillations of monotone functions assigned to sub-objects~\cite{ChEW_pseudotree, Todorcevic_Roscillate, Bartosova2025},
study of the partial order of ages (ranks or orbits) of vertices~\cite{sauer2003canonical}, and arguments based on the distance and diameter in metric spaces~\cite{Laflamme2006}.

Solving the question about finiteness of big Ramsey degrees of $\str{R}^2_\omega$ suggests the following question about its reduct, which forgets the actual labels of edges and only records information about pairs of vertices with equivalent labels:
\begin{problem}
	Let $L$ be a relational language with a single quaternary relation $R$ and $\mathcal K$ the class of all finite $L$-structures
	$\str{A}$ such that 
	\begin{enumerate}
		\item for every $(a,b,c,d)\in R^\str{A}$ it holds that $a\neq b$, $c\neq d$ and $(c,d,a,b)\in R^\str{A}$,
		\item for pair of distinct vertices $a,b$ of $\str{A}$ it holds that $(a,b,a,b),(a,b,b,a)\in R^\str{A}$,
		\item whenever $(a,b,c,d)$ and $(c,d,e,f)$ is in $R^\str{A}$ then also $(a,b,e,f)\in  R^\str{A}$.
	\end{enumerate}
	(In other words, $R^\str{A}$ defines an equivalence on 2-element subsets of vertices of $A$.)
	Does the \Fraisse{} limit of $\mathcal K$ have finite big Ramsey degrees?
\end{problem}
It is known that $\mathcal K$ has a precompact Ramsey expansion~\cite{Hubicka2016,hubicka2025twenty} (fixing a linear ordering of vertices as well as a linear ordering of equivalence classes) and thus finite small Ramsey degrees. However, the question about the finiteness of big Ramsey degrees is fully open.
\section{Compressed tree of types}
We devote the rest of this abstract to a discussion of the proof of Theorem~\ref{thm:main}.
Toward that, we fix $u>1$ and a hypergraph $\str{R}^u_\omega$ with vertex set $\omega$ (that is, we work with an arbitrary but fixed enumeration of $\str{R}^u_\omega$). 
We will construct explicit colourings which contradict the existence of big Ramsey degrees in $\str{R}^u_\omega$.

Our construction is based on ideas used for analyzing structures which \emph{do} have finite big Ramsey degrees. This is done using Ramsey-type theorems working with the so-called \emph{tree of types}, see e.g.~\cite{hubicka2024survey}.  The main difficulty of applying this technique to $\str{R}^u_\omega$ is the fact that the tree of types of $\str{R}^u_\omega$ is infinitely branching.  We overcome this problem by using a related tree which is finitely branching but the number
of immediate successors of a vertex grows very rapidly. This lets us reverse the argument and instead of showing that big Ramsey degrees are finite, we obtain enough structure to show that they are infinite.

Let us introduce the key definitions. Put $L=\omega\cup \{\star\}$ where $\star$ will play the role of a special label which intuitively means that the information is ``missing''.
\begin{definition}[$f$-type]
	Let $f\colon \omega\to\omega\cup \{\omega\}$ be an arbitrary function.  We call an $L$-edge-labeled $u$-uniform hypergraph $\str{X}$ an \emph{$f$-type of level $\ell$} if:
	\begin{enumerate}
		\item The vertex set of $\str{X}$ is $X=\{0,1,\ldots,\ell-1\}\cup \{t\}$ where $t$ is a special vertex, called the \emph{type vertex}.
		\item For every $E\in \binom{X}{u}$ with $e_\str{X}(E)\neq \star$ it holds that $t\in E$ and $e_\str{X}(E)<f(\max (E\setminus \{t\})).$ 
		\item For every $E\in \binom{X}{u}$ with $t\in E$ such that $f(\max (E\setminus \{t\}))=\omega$ it holds that $e_\str{X}(E)\neq \star$.
	\end{enumerate}
	We also call a hyper-graph $\str{X}$ simply an \emph{$f$-type} if it is an $f$-type of level $\ell$ for some $\ell\in \omega$. In this situation we put $\ell(\str{X})=\ell$.
\end{definition}
\begin{definition}[Tree of $f$-types]
	Let $f\colon \omega\to\omega\cup \{\omega\}$ be an arbitrary function. By $T_f$ we denote the set of all $f$-types. 
	We will view $T_f$ as a (set-theoretic) tree equipped with a partial order $\sqsubseteq$ and operation $\meet$ (\emph{meet}) defined as follows:
	Given $f$-types $\str{X},\str{Y}\in T_f$ we put $\str{X}\sqsubseteq \str{Y}$ if and only if $\str{X}$ is an (induced) sub-structure of $\str{Y}$.
	By $\str{X}\meet\str{Y}$ we denote the (unique) $f$-type $\str{Z}\in T_f$ such that $\str{Z}\sqsubseteq \str{X}$, $\str{Z}\sqsubseteq \str{Y}$ of largest level among all $f$-types with this property.
	Finally, given integer $\ell$, we put $T_f(\ell)=\{\str{X}\in T_f:\ell(\str{X})=\ell\}$
	and call it the \emph{level $\ell$} of $T_f$.  We call $\str{X}\in T_f$ an \emph{immediate successor} of $\str{Y}\in T_f$ if and only if $\str{Y}\sqsubseteq \str{X}$ and $\ell(\str{X})=\ell(\str{Y})+1$.
\end{definition}

	The usual tree of types corresponds to using the constant function $f^\omega$ where $f^\omega(i)=\omega$ for every $i\in \omega$.  Every $f^\omega$-type $\str{X}$ of level $\ell$ can be thought of as a one vertex extension
	of some $\omega$-edge-labeled $u$-uniform hypergraph $\str{A}$ with vertex set $\{0,1,\ldots,\ell-1\}$. For this reason we put $e_\str{X}(E)=\star$ for every $E\in \binom{\{0,1,\ldots, \ell-1\}}{u}$ since this label is determined by $\str{A}$. We will consider functions $f$ with $\mathrm{Im}(f)\subseteq \omega$ and then $f$-types capture only partial information about these one vertex extensions.  We make this explicit as follows:

\begin{definition}[$f$-type of a vertex]
	Given $v\in \str{R}^u_\omega$, the \emph{$f$-type of $v$}, denoted by $\tp{f}{v}$, is an $f$-type $\str{X}$ of level $v$ where given $E\in \binom{X}{u}$, and writing $E'=(E\setminus \{t\})\cup \{v\}$, we have 
	$$e_\str{X}(E)=
	  \begin{cases}
		  e_{\str{R}^u_\omega}(E')&\hbox{if }t\in E\hbox { and }e_{\str{R}^u_\omega}(E')<f(\max (E\setminus \{t\}))\\
		  \star &\mathrm{otherwise.}\\
	  \end{cases}$$
\end{definition}
	Notice that for every choice of $f$ it follows by universality and homogeneity of $\str{R}^u_\omega$ that for every $f$-type $\str{X}$ there exist infinitely many vertices $v$ of $\str{R}^u_\omega$ satisfying $\tp{f}{v}\sqsupseteq\str{X}$.

\section{Persistent colouring of $\str{R}^u_\omega$}
	If function $f\colon \omega\to\omega\cup\{\omega\}$ is fixed then every vertex $v$ of $\str{R}^u_\omega$ is associated with the $f$-type $\tp{f}{v}\in T_f$. Given two vertices of $\str{R}^u_\omega$, we can then study their iterated meet closure in the tree $T_f$ defined as follows.
\begin{definition}
	\label{def:height}
	Given a pair of nodes $\str{X},\str{Y}\in T_f$, its \emph{$f$-height}, denoted by $\height{f}{\str{X}}{\str{Y}}$, is the number of repetitions of the following procedure:
	\begin{enumerate}
		\item Put $\str{Z}=\str{X}\meet \str{Y}$.
		\item If $\tp{f}{\ell(\str{Z})}=\str{Z}$ terminate.
		\item Repeat from step 1 with $\str{X}=\tp{f}{\ell(\str{Z})}$ and $\str{Y}=\str{Z}$.
	\end{enumerate}
	Given vertices $v,w\in R^u_\omega$ we also put $\height{f}{v}{w}=\height{f}{\tp{f}{v}}{\tp{f}{w}}.$
\end{definition}
\begin{theorem}
	\label{thm:main2}
	Assume that $f(\ell)\colon \omega\to\omega$ is a function satisfying 
	$$f(\ell)\geq \prod_{u-2\leq i<\ell} (f(i)+1)^{\binom{i}{u-2}}$$
	for every $\ell\in \omega$.
	Then for every embedding $\varphi \colon \str{R}^u_\omega\to \str{R}^u_\omega$ there exists integer $m$ such that for every $n>m$ there exist vertices $v,w\in \varphi[R^u_\omega]$ satisfying
	\begin{enumerate}
		\item if $u=2$ then $e_{\str{R}^u_\omega}(\{v,w\})=0$ and,
		\item $\height{f}{v}{w}=n$.
	\end{enumerate}
\end{theorem}
Notice that Theorem~\ref{thm:main2} immediately implies Theorem~\ref{thm:main}. Let $\str{A}$ be as in Theorem~\ref{thm:main} and assume $A=\{0,1\}$. If $u=2$, without loss of generality we can also assume that $e_\str{A}(\{0,1\})=0$. Given finite $n>1$, we define colouring $\chi_n\colon \Emb(\str{A},\str{R}^\omega_u) \to n$ by putting $\chi_n(h)=\height{f}{h(0)}{h(1)}\mod n$ for every $h\in \Emb(\str{A},\str{R}^u_\omega)$. By Theorem~\ref{thm:main2}, for every embedding $\varphi \colon \str{R}^u_\omega\to \str{R}^u_\omega$ there are copies of $\str{A}$ in every colour showing that the big Ramsey degree of $\str{A}$ is greater than $n$.

\begin{proof}[Proof of Theorem~\ref{thm:main2} (sketch)]
	\begin{figure}
		\centering
		\includegraphics{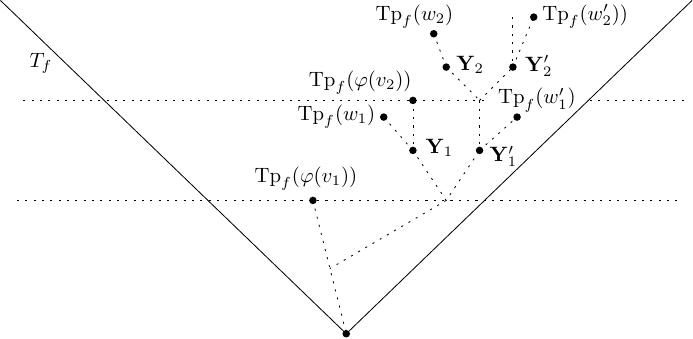}
		\caption{Configuration of tree nodes used in the proof of Theorem~\ref{thm:main2}.}
		\label{fig:2b}
	\end{figure}
	Fix $f$ and embedding $\varphi \colon \str{R}^u_\omega\to \str{R}^u_\omega$ as in the statement.  The rapid growth of $f$ ensures that for every $f$-type $\str{X}$ it holds that
	number of immediate successors of $\str{X}$ is greater than number of nodes of $T_f$ of level $\ell(\str{X})$.  This makes it possible to obtain for every vertex $v\in R_\omega^u$ (up to $u-1$ exceptions) vertices $v_+,v'_+\in R^u_\omega$ with the property that $\varphi(v)=\ell(\tp{f}{\varphi(v)})=\ell(\tp{f}{\varphi(v_+)}\meet \tp{f}{\varphi(v'_+)}$. 

	Using a technique inspired by Lachlan, Sauer, and Vuksanovic~\cite{Laflamme2006} we obtain vertices $v_0, v_1,\ldots \in R^u_\omega$, $w_1,w_2,\ldots \in \varphi[R^u_\omega]$, $w'_1,w'_2,\ldots\in \varphi[R^u_\omega]$ and nodes $\str{Y}_0,\str{Y}_1,\ldots$, $\str{Y}'_0,\str{Y}'_1,\ldots$ in configuration as depicted in Figure~\ref{fig:2b}. Then it follows that for every $i>1$ we get $\height{f}{w_{i+1}}{w'_{i+1}} = \height{f}{w_{i}}{w'_{i}}+1$.
\end{proof}

\section{Acknowledgements}
The initial variant of this construction was obtained during the  Thematic Program on Set Theoretic Methods in Algebra, Dynamics and Geometry at Fields Institute. We would like to
thank the organizers  Michael Hrušák, Kathryn Mann, Justin Moore and Stevo Todorcevic, for such a great program.  



\bibliographystyle{plain}
\bibliography{ramsey}
\end{document}